\documentclass[12pt, reqno]{amsart}
\usepackage{amssymb, amsthm, amsmath, amsfonts}
\usepackage{array, epsfig}
\usepackage{bbm}

  \evensidemargin
\oddsidemargin
\begin{document}

\renewcommand{\proofname}{\bf Proof}
\newtheorem{rem}{Remark}
\newtheorem*{cor*}{Corollary}
\newtheorem{cor}{Corollary}
\newtheorem{prop}{Proposition}
\newtheorem{lem}{Lemma}
\newtheorem{theo}{Theorem}
\newfont{\zapf}{pzcmi}

\def\R{\mathbb{R}}
\def\N{\mathbb{N}}
\def\E{\mathbb{E}}
\def\P{\mathbb{P}}
\def\V{\mathbb{D}}
\def\I{\mathbbm{1}}
\newcommand{\D}{\hbox{\zapf D}}

\title[On the probability that integrated random walks stay positive]
{On the probability that integrated random walks stay positive}

\author[V. Vysotsky]{Vladislav Vysotsky}
\thanks{Supported in part by the Moebius Contest Foundation for Young Scientists.}
\email{vysotsky@math.udel.edu}
\address{
Mathematical Sciences\\
University of Delaware\\
517A Ewing Hall\\
Newark, DE 19716, USA}

\begin{abstract}
Let $S_n$ be a centered random walk with a finite variance, and consider the sequence $A_n:=\sum_{i=1}^n S_i$, which we call an {\it integrated random walk}. We are interested in the asymptotics of $$p_N:=\P \Bigl \{ \min \limits_{1 \le k \le N} A_k \ge 0 \Bigr \}$$ as $N \to \infty$. Sinai (1992) proved that $p_N \asymp N^{-1/4}$ if $S_n$ is a simple random walk. We show that $p_N \asymp N^{-1/4}$ for some other types of random walks that include double-sided exponential and double-sided geometric walks, both not necessarily symmetric. We also prove that $p_N \le c N^{-1/4}$ for integer-valued walks and upper exponential walks, which are the walks such that $\mbox{Law} (S_1 | S_1>0)$ is an exponential distribution.
\end{abstract}

\maketitle

\section{Introduction}
Let $S_n$ be a centered random walk with a finite variance, and consider the sequence of r.v.'s $A_n:=\sum_{i=1}^n S_i$, which we  call an {\it integrated random walk}. We are interested in the asymptotical behavior of the probabilities $$p_N:=\P \Bigl \{ \min \limits_{1 \le k \le N} A_k  \ge 0 \Bigr \}$$ as $N \to \infty$. We came to this problem while studying properties of so-called sticky particle systems, see Vysotsky~\cite{Me}.
One may consider this question as a particular case of the general problem on finding one-sided small deviation probabilities of a random sequence.

The only known sharp result on $p_N$ is due to Sinai~\cite{Sinai}, who showed that $p_N \asymp N^{-1/4}$ for a simple random walk. Sinai studied this problem in connection with solutions of the Burgers equation with random initial data. Caravenna and Deuschel~\cite{Polymers} considered such probabilities in relation to random polymers, and they obtained a rough non-polynomial upper bound for $p_N$ for general random walks. A rough lower bound is given by the trivial $p_N \ge \P \bigl \{ \min \limits_{1 \le k \le N} S_k  \ge 0 \bigr \} \sim c N^{-1/2}$.

For the continuous version of the problem,
\begin{equation}\label{int W}
\P \biggl \{ \min_{0 \le s \le N} \int_0^s W(u) du \ge - 1 \biggr \} \sim c N^{-1/4},
\end{equation}
where $W(u)$ is a Wiener process and $c$ is a positive constant that could be found explicitly. This result of Isozaki and Watanabe~\cite{Japan} refines a weaker version of \eqref{int W} obtained by Sinai~\cite{Sinai}, who had $\asymp$ instead of $\sim$ in the right-hand side. Isozaki and Watanabe actually conclude \eqref{int W} from McKean~\cite{McKean}.

These asymptotical results of \cite{Japan} and \cite{Sinai} prompted the author to conjecture in \cite{Me} that $p_N \asymp N^{-1/4}$ for any centered random walk with a finite variance. In this paper we obtain several results that partially prove the conjecture. Note that it seems impossible to get the relation $p_N \asymp N^{-1/4}$ directly from \eqref{int W} because even if $S_n = W(n)$ is a standard Gaussian random walk, $\int_0^n W(u) du -\sum_{i=1}^n W(i)$ has order $n^{1/2}$.

Let us first state a result on the upper bound for $p_N$. We say that a r.v. $X$ is {\it upper exponential} if $\mbox{Law} (X | X>0)$ is an exponential  distribution. A typical example is an exponential r.v. centered by its expectation. An integer-valued r.v. $X$ is called {\it upper geometric} if $\mbox{Law} (X | X>0)$ is a geometric distribution. In what follows, we refer to random walks by the type of common distribution of their increments.

\begin{theo}
\label{UPPER B}
Let $S_n$ be a centered random walk with a finite variance that is either
integer-valued or upper exponential. Then $p_N \le c N^{-1/4}$ for some constant $c>0$.
\end{theo}

Our proof is based on the fact that any integer-valued random walk $S_n$ with $\E S_1=0$ and $Var (S_1) < \infty$ returns to zero almost surely. This of course does not hold for the ``continuous'' case, and we need to impose the condition of upper exponentiality. It is unclear if it is possible to remove this additional assumption using discretization and the result for integer-valued walks: the discretized centered walk should be also centered. On the other hand, it worth to cite the comment from Feller~\cite[p. 404]{Feller}: ``At first sight the distribution $F$ [an upper exponential distribution] ... appears artificial, but the type turns up frequently in connection with Poisson processes, queuing theory, ruin problems, etc.'' Moreover, Theorem~\ref{UPPER B} is important for the results of~\cite{Me}, where the primary interest was in exponential walks centered by expectation.

We prove lower bounds for $p_N$ under more restrictive conditions, which are imposed on $\mbox{Law} (S_1 | S_1 < 0)$. A r.v. $X$ is called {\it two-sided exponential} if both $X$ and $-X$ are upper exponential. A typical example is the Laplace distribution but two-sided exponential distributions are not necessarily symmetric. Further, we follow Spitzer~\cite{Spitzer} and say that a r.v. $X$ is {\it right-continuous} if $\P \bigl \{ X \in \{\dots, -1, 0, 1\} \bigr \} =1$. Finally, define a {\it slackened simple random walk} as a nondegenerate symmetric right-continuous walk. Informally speaking, these are simple random walks allowed to stay immobile.

Note that upper exponential, upper geometric, and right-continuous random walks have the same common property, which plays the key role in our proofs: the overshoot over any fixed level is independent of the moment when its occurs and also of the trajectory of the walk up to this moment.

\begin{theo} \label{LOWER B}
1. Let $S_n$ be a centered random walk such that both $S_n$ and $-S_n$ are either upper geometric or right-continuous. Then $N^{-1/4} l(N) \le p_N $ for some function $l(n)$ that is slowly varying at infinity.

2. Let $S_n$ be a centered random walk that is either double-sided exponential or satisfies conditions of Part 1 and is symmetric. Then $c N^{-1/4} \le p_N $ for some constant $c>0$.
\end{theo}

Note that Part 1 covers walks that are lower geometric and right-continuous or vise versa, and both Parts 1 and 2 cover walks with $\P\{S_1 = 0\} >0$. From Theorems~\ref{UPPER B} and \ref{LOWER B}, we conclude the following.

\begin{cor*}
Let $S_n$ be a centered random walk that is two-sided exponential, slackened simple, or  symmetric two-sided geometric. Then $p_N \asymp N^{-1/4}$.
\end{cor*}

We prove the upper bound following the main idea of the proof of Sinai~\cite{Sinai}, although we make significant simplifications. For the lower bound, only a sketch of the proof was given in~\cite{Sinai} but all interesting details were omitted. We failed to conclude these missing arguments, and therefore we prove the lower bounds in an entirely different way. In fact, \cite{Sinai} implicitly uses a local limit theorem for bivariate walks whose first component is conditioned to stay positive and, as the main difficulty, has increments from the domain of attraction of an $\alpha$-stable law (with $\alpha=1/3$). It was only recently when Vatutin and Wachtel~\cite{Vatutin} proved a weaker result, a local limit theorem for such heavy-tailed (univariate) walks conditioned to stay positive. Thus, the other contribution of our paper is the first complete proof of the lower bound for $p_N$.

The paper is organized as follows. In Section~\ref{Sec Heur -> Proof} we give a heuristic explanation of why $p_N \asymp N^{-1/4}$ for a simple random walk, and then develop and generalize the basic idea of this heuristic approach making it applicable to the random walks considered here. In Section~\ref{Sec Properties} we prove preparatory results on durations and areas of ``cycles'' of random walks; a {\it cycle} is a positive excursion together with the consecutive negative excursion. In particular, in Proposition~\ref{TAILS} we find the asymptotics of the ``tail'' of the joint distribution of these variables. This simplifies and generalizes the analogous result of Sinai~\cite{Sinai} obtained by sophisticated but tedious arguments which work only for simple random walks. In Sections~\ref{Sec Upper} and~\ref{Sec Lower} we prove upper and lower bounds for $p_N$, respectively. Finally, in Section~\ref{Sec Concluding} we make concluding remarks and discuss possible ways to prove the lower bound under less restrictive conditions.

\section{From heuristics to proofs} \label{Sec Heur -> Proof}
\subsection{Heuristics for the asymptotics of $p_N$} \label{SSec Heur}
Let us give a heuristic explanation of why $p_N \asymp N^{-1/4}$ for a simple random walk. We took  the following arguments from the survey paper Vergassola et al.~\cite{Verga}, which provides a simple informal explanation of the complicated proofs of Sinai~\cite{Sinai}. The approach itself was introduced in~\cite{Sinai} although $p_N$ was estimated there in a different way.

The main idea of Sinai's method is to decompose the trajectory of the random walk $S_k$
into  independent excursions. Define the moments of hitting zero as $\tau_0^0:=0$ and
$\tau_{n+1}^0:= \min \bigl\{ k > \tau_n^0: S_k = 0 \bigr\}$ for $n \ge 0$. Let
$\theta_n^0 := \tau_n^0- \tau_{n-1}^0$ be durations of excursions, let $\xi_n^0 :=
\sum_{i=\tau_{n-1}^0 + 1}^{\tau_n^0} S_i$ be their areas, and let $\eta^0(N)$ be the
number of complete excursions by the time $N$, namely, $\eta^0(N):= \max \bigl\{k \ge 0:
\tau_k^0 \le N \bigr\} = \max \bigl\{k \ge 0: \sum_{i=1}^k \theta_i^0 \le N \bigr\}$.
Since for each $n$ it holds that $$\Bigl \{ \min \limits_{1 \le k \le \tau_n^0} \sum_{i=1}^k S_i  \ge 0 \Bigr \} = \Bigl \{ \min \limits_{1 \le k \le n } \sum_{i=1}^k \xi_i^0 \ge 0 \Bigr \},$$ as $\tau_{\eta^0(N)}^0 \le N < \tau_{\eta^0(N)+1}^0$, we have
\begin{equation}\label{main estimate}
\P \Bigl \{ \min \limits_{1 \le k \le \eta^0(N)+1} \sum_{i=1}^k \xi_i^0
\ge 0 \Bigr \} \le \P \Bigl \{ \min \limits_{1 \le k \le N} \sum_{i=1}^k S_i  \ge 0 \Bigr \} \le \P \Bigl \{ \min \limits_{1 \le k \le \eta^0(N)} \sum_{i=1}^k \xi_i^0  \ge 0 \Bigr \}.
\end{equation}

Note that $\xi_n^0$ are i.i.d. and symmetric, hence $\sum_{i=1}^k \xi_i^0$ is a symmetric random walk. It is well known that for such random walks $$ \P \Bigl \{ \min \limits_{1
\le k \le n} \sum_{i=1}^k \xi_i^0  \ge 0 \Bigr \} \sim \frac{c}{\sqrt{n}}$$ as $n \to
\infty$ for a certain constant $c>0$. On the other hand, $\eta^0(N) \asymp N^{1/2}$ in probability as $N \to \infty$ because of another well-known fact that $\theta_1^0$ belongs to the domain of normal attraction of an $\alpha$-stable law with exponent $1/2$. Were $\eta^0(N)$ independent with the walk $\sum_{i=1}^k \xi_i^0$, these asymptotical
estimates and \eqref{main estimate} would immediately imply $p_N \asymp N^{-1/4}$.

Unfortunately, $\eta^0(N) = \max \bigl\{k \ge 0: \sum_{i=1}^k \theta_i^0 \le N \bigr\}$
and $\sum_{i=1}^k \xi_i^0$ are dependent, and a careful study of the joint distributions
of $(\xi_1^0, \theta_1^0)$ is  required. Sinai~\cite{Sinai} gives a tedious analysis of
the generating function of $(\xi_1^0, \theta_1^0)$ using the theory of continuous
fractions. However, these arguments can not be generalized since the crucial recursive
relation for the generating function of $(\xi_1^0, \theta_1^0)$ was obtained in~\cite{Sinai} using binary structure of increments of simple random walks.

\subsection{Preparatory definitions}
In our proofs, we use a generalization of the described approach of decomposing the trajectory of the walk into independent excursions. In this section we introduce appropriate definitions.

Suppose, at first, that $S_n$ is an integer-valued random walk. We keep the previous notations but define $\tau_n^0$ as the moments of {\it returning} to zero: $\tau_0^0:=0$ and $\tau_{n+1}^0:= \min \bigl\{ k > \tau_n^0 +1: S_k = 0, S_{k-1} \neq 0 \bigr\}$ for $n \ge 0$, which coincide with the moments of {\it hitting} zero if $S_n$ is a simple random walk. The variables $\tau_{n+1}^0$ are finite with probability $1$ because the walk is integer-valued, centered, and has a finite variance. Only the upper bound in \eqref{main estimate} remains valid because the walk can jump over the zero level without hitting it.

Clearly, the described approach does not work for general walks. We shall consider different stopping times.

Define conditional probability $\widetilde{\P} \{ \cdot \} := \P \{ \cdot  | S_1 > 0 \}$
and define $\tilde{p}_N$ as $p_N$ but with $\P$ replaced by $\widetilde{\P}$. Note that
it suffices to prove Theorems~\ref{UPPER B} and \ref{LOWER B} for $\tilde{p}_N$ instead of $p_N$. Indeed, $$p_N = \P \Bigl \{ \min \limits_{1 \le k \le N} \sum_{i=1}^k S_i  \ge 0 \Bigr \} = a_+ \sum_{n=0}^N a_0^n \, \P \Bigl \{ \min \limits_{1 \le k \le N-n} \sum_{i=1}^k S_i \ge 0 \Bigl | S_1>0\Bigr. \Bigr \} = a_+ \sum_{n=0}^N
a_0^n \, \tilde{p}_{N-n},$$ where $$a_+:= \P \{ S_1 >0\}, \quad a_0:= \P \{ S_1 =0\},
\quad a_-:= \P \{ S_1 <0\}.$$ Hence
\begin{equation}
\label{p eqv p}
p_N \asymp \tilde{p}_N
\end{equation}
if $\tilde{p}_N$ decays polynomially.

Now, let $X_1^+$ be a r.v. with the distribution $\mbox{Law} (S_1 | S_1>0)$ and
independent with the walk $S_n$, and put $\widetilde{S}_n:=X_1^+ + S_n - S_1$ for $n \ge
1$. Clearly, $$\mbox{Law} (\widetilde{S}_1 , \widetilde{S}_2, \dots) = \mbox{Law}
(S_1, S_2, \dots | S_1>0).$$

\begin{figure}[ht]
\centering
\includegraphics[width=1.0\textwidth]{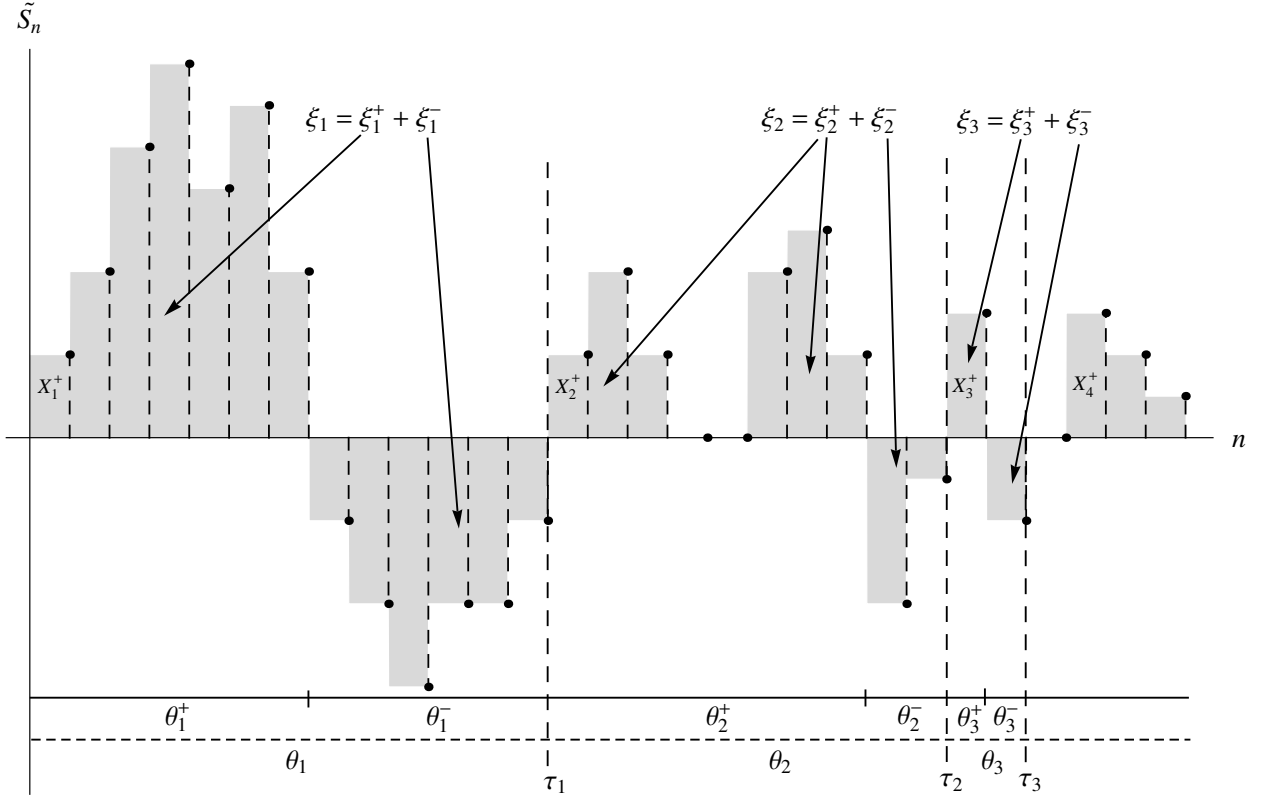}
\caption{Decomposition of the trajectory of $\widetilde{S}_n$ into ``cycles''.} \label{Pic 1}
\end{figure}

For convenience of the reader, the following definitions are represented in comprehensive Fig.~\ref{Pic 1}. Define the moments $\tau_n$ when $\widetilde{S}_k$ overshoots the zero level {\it from below}: $\tau_0:=0$ and $\tau_{n+1}:= \max \bigl\{ k > \tau_n: \widetilde{S}_k \le 0 \bigr\}$ for $n \ge 0$. It is readily seen that $\tau_n +1$ are stopping times. Denote $\theta_n := \tau_n- \tau_{n-1}$ and $\xi_n := \sum_{i=\tau_{n-1}+1}^{\tau_n} \widetilde{S}_i$, and let $\eta(N)$ be the number of overshoots of the zero level from below by the time $N$, namely, $$\eta(N):= \max \bigl\{k: \tau_k \le N \bigr\} = \max \bigl\{k: \sum_{i=1}^k \theta_i \le N \bigr\}.$$ Now, by analogy with \eqref{main estimate}, we write
\begin{equation} \label{main estimate'}
\P \Bigl \{ \min \limits_{1 \le k \le \eta(N)+1} \sum_{i=1}^k
\xi_i  \ge 0 \Bigr \} \le \widetilde{\P} \Bigl \{ \min \limits_{1 \le k \le N}
\sum_{i=1}^k S_i  \ge 0 \Bigr \} \le \P \Bigl \{ \min \limits_{1 \le k \le \eta(N)} \sum_{i=1}^k \xi_i  \ge 0 \Bigr \}.
\end{equation}

It is clear that the moments of overshoots $\tau_n$ partition the trajectory of $\widetilde{S}_k$ into ``cycles'' that consist of one weak positive and the consequent weak negative excursion (that is, nonnegative and nonpositive, respectively, but we will omit ``weak'' in what follows). Let $\theta_n^+ := \max \bigl\{ k > 0: \widetilde{S}_{\tau_{n-1}+k} \ge 0 \bigr\}$ and $\theta_n^- := \max \bigl\{ k > 0: \widetilde{S}_{\tau_{n-1}+\theta_n^+ + k} \le 0 \bigr\}$ be the lengthes and let $\xi_n^+ := \sum_{i=\tau_{n-1}+1}^{\tau_{n-1}+\theta_n^+} \widetilde{S}_i$ and $\xi_n^- :=
\sum_{i=\tau_{n-1}+\theta_n^+ + 1}^{\tau_n} \widetilde{S}_i$ be the areas of these excursions, respectively; obviously, $\xi_n=\xi_n^+ + \xi_n^-$ and $\theta_n=\theta_n^+ + \theta_n^-$. The following observation plays the key role in our paper.

\begin{lem} \label{IID}
Let $S_n$ be a centered random walk with a finite variance.

(a) If $S_n$ is integer-valued, then random vectors $(\xi_n^0, \theta_n^0)_{n \ge 1}$ are i.i.d.

(b) If $S_n$ is upper exponential, upper geometric, or right-continuous, then the random vectors $(\xi_n, \theta_n)_{n \ge 1}$ are i.i.d., $(\xi_n^+, \theta_n^+)_{n \ge 1}$ are i.i.d., and $(\xi_n^-, \theta_n^-)_{n \ge 1}$ are i.i.d. If, in addition, $S_n$ satisfies assumptions of Theorem~\ref{LOWER B}, then $(\xi_n^+, \theta_n^+)_{n \ge 1}$ and $(\xi_n^-, \theta_n^-)_{n \ge 1}$ are mutually independent.
\end{lem}

Note: from this point on, $S_n$ satisfies assumptions of Theorem~\ref{LOWER B} means that it satisfies assumptions of Part 1 or Part 2 of the theorem. The lemma, basically, shows that under the made assumptions, the cycles of the walk are i.i.d.

\begin{proof}
Part (a) is trivial. For Part (b), note that the overshoots over the zero level $X_n^+:=\widetilde{S}_{\tau_{n-1}+1}$ are i.i.d. and their common distribution is $\mbox{Law}(S_1|S_1>0)$, which is exponential, geometric, or $\delta_1$. This naturally follows from the memoryless property of these distributions; a proof could be found in Example XII.4(a) from Feller~\cite{Feller}. In the same way, we show that $X_n^+$ are independent from the ``past'' $\widetilde{S}_1, \dots, \widetilde{S}_{\tau_{n-1}}$. Now from $\xi_n =  \sum_{i=\tau_{n-1}+1}^{\tau_n} \widetilde{S}_i = \sum_{i=\tau_{n-1}+1}^{\tau_n} (X_n^+ + \widetilde{S}_i - \widetilde{S}_{\tau_{n-1}+1})$ and $\theta_n = \max \bigl\{ k > 0: X_n^+ + \widetilde{S}_{\tau_{n-1}+k} - \widetilde{S}_{\tau_{n-1}+1} \le 0 \bigr\}$ we see that $(\xi_n, \theta_n)$ are i.i.d. as $\tau_n +1$ are stopping times. The proof of the other statements is analogous.
\end{proof}

\section{Areas and durations of excursions and cycles} \label{Sec Properties}
We already explained in Sec.~\ref{SSec Heur} why it is important to study properties of the joint distribution of $\xi_1$ and $\theta_1$. Here we prove several crucial results on $(\xi_1, \theta_1)$, $(\xi_1^+, \theta_1^+)$, $(\xi_1^-, \theta_1^-)$, and $(\xi_1^0, \theta_1^0)$, which are used in the proofs of Theorems~\ref{UPPER B} and~\ref{LOWER B}.

We start with a surprising lemma which allows us, in certain cases, to reduce a complicated study of the joint distribution of $(\xi_1, \theta_1)$ to a much simpler consideration of its marginal distributions.

\begin{lem}
\label{EXP SYMM}
Let $S_n$ be a centered random walk with a finite variance. If $S_n$ is upper exponential, then the distribution of $\xi_1$ is symmetric, and moreover, $(\xi_1, \theta_1) \stackrel{\D}{=} (-\xi_1, \theta_1)$ and $(\xi_1^+, \theta_1^+, \xi_1^-, \theta_1^-) \stackrel{\D}{=} (-\xi_1^-, \theta_1^-, -\xi_1^+, \theta_1^+)$. If $S_n$ is integer-valued, then the distribution of $\xi_1^0$ is symmetric, and moreover, $(\xi_1^0, \theta_1^0) \stackrel{\D}{=} (-\xi_1^0, \theta_1^0)$.
\end{lem}
\begin{proof}
Let us start with the upper exponential case assuming, without loss of generality, that $\mbox{Law} (X | X>0)$ is a standard exponential distribution. Since $\xi_1 = \widetilde{S}_1 + \dots + \widetilde{S}_{\theta_1}$, it suffices to show that for each $i,j \ge 1$, the measures $\P \bigl \{(\widetilde{S}_1, \dots , \widetilde{S}_{\theta_1}) \in \cdot \, , \theta_1^+ =i, \theta_1^-=j \bigr\}$ and $\P \bigl\{(-\widetilde{S}_{\theta_1}, \dots, -\widetilde{S}_1) \in \cdot \, , \theta_1^+ =j, \theta_1^-=i \bigr\}$ coincide. This statement follows from the observation that for any $x_1, \dots, x_i >0$ and $x_{i+1}, \dots, x_{i+j} < 0$,
\begin{eqnarray*}
&&\P \bigl \{\widetilde{S}_1 \in d x_1, \dots, \widetilde{S}_{i+j} \in d x_{i+j}, \theta_1^+ =i, \theta_1^-=j \bigr\} \\
&=& a_+ e^{x_{i+j}-x_1} \E \bigl \{S_2 \in d x_2, \dots, S_{i+j-1} \in d x_{i+j-1} \bigl| \bigr. S_1 = x_1, S_{i+j} = x_{i+j} \bigr\}  d x_1 d x_{i+j}
\end{eqnarray*}
and
\begin{eqnarray*}
&&\P \bigl \{\widetilde{S}_{i+j} \in - d x_1, \widetilde{S}_{i+j-1} \in -d x_2, \dots, \widetilde{S}_1 \in  - d x_{i+j}, \theta_1^+ =j, \theta_1^-=i \bigr\} \\
&=& a_+ e^{x_{i+j}-x_1} \E \bigl \{S_2 \in -d x_{i+j-1}, \dots, S_{k-1} \in -d x_2 \bigl| \bigr. S_1 = - x_{i+j}, S_{i+j} = -x_1 \bigr\} d x_1 d x_{i+j} .
\end{eqnarray*}
Indeed, the conditional expectations in the right hand sides coincide for any random walk: this is, essentially, the well-known property of duality of random walks.

The proof for the lattice case is analogous: since $\xi_1^0 = S_1 + \dots +
S_{\theta_1^0}$, use that for any $i \ge 0, j \ge 1$, and any integer $x_{i+1}, \dots, x_{i+j} \neq 0$, it holds that
\begin{eqnarray*}
&&\P \bigl \{S_1 = \dots = S_i=0, S_{i+1} =x_{i+1}, \dots, S_{i+j} = x_{i+j} , S_{i+j+1} = 0 \bigr\} \\
&=& \P \bigl \{S_1 = \dots = S_i=0, S_{i+1} =-x_{i+j}, \dots, S_{i+j} = -x_{i+1} , S_{i+j+1} = 0\bigr\}
\end{eqnarray*}
for any random walk.
\end{proof}

Note that the distribution of $\xi_1$ is not symmetric even for two-sided geometric random walks unless $a_- = a_+$. The proof presented above for the upper exponential case does not work here because two-sided geometric walks can return to zero.

In order to state the next result, recall that r.v.'s $Y_1, \dots, Y_k$ are {\it associated} if $$cov \bigl(f(Y_1, \dots, Y_k), \,g(Y_1, \dots, Y_k) \bigr) \ge 0$$ for any coordinate-wise nondecreasing functions $f,g: \R^k \to \R$ such that the covariance is well defined. An infinite set of r.v.'s is associated if any finite subset of its
variables is associated. The following sufficient conditions of association are well
known, see Esary~et~al.~\cite{Esary}:
\begin{enumerate}
\item[(a)] A set consisting of a single r.v. is associated.

\item[(b)] Independent r.v.'s are associated.

\item[(c)]  Coordinate-wise nondecreasing functions (of a finite number of
variables) of associated r.v.'s are associated.

\item[(d)]  If $Y_{1,u}, \dots, Y_{k,u}$ are associated for every $u$ and $(Y_{1,u}, \dots, Y_{k,u}) \stackrel{\D}{\longrightarrow} (Y_1, \dots,
Y_k)$ as $u \to \infty$, then $Y_1, \dots, Y_k$ are associated.

\item[(e)]  If two sets of associated variables are independent, then the
union of these sets is also associated.
\end{enumerate}

We now state the other result that allows us, in some cases, to proceed from study of the joint distribution of $(\xi_1, \theta_1)$ to a consideration of the distributions of $\xi_1$ and $\theta_1$.

\begin{lem}
\label{ASSOC}
Under assumptions of Theorem~\ref{LOWER B}, the random
variables $\{ \xi_n, \theta_n^+\}_{n \ge 1}$ are associated.
\end{lem}

\begin{proof}
We first show that $\xi_1^+$ and $\theta_1^+$ are associated. Indeed, by (b) and (c), the r.v.'s $\sum_{i=1}^{\min\{k, \theta_1^+\}} \widetilde{S}_i$ and $\min\{k, \theta_1^+\}$ are associated for each $k$ as coordinate-wise nondecreasing functions of the first $k$
independent increments of the walk. Since $\bigl (\sum_{i=1}^{\min\{k, \theta_1^+\}} \widetilde{S}_i, \min\{k, \theta_1^+\} \bigr ) \to \bigl (\xi_1^+, \theta_1^+ \bigr )$ with probability 1 as $k \to \infty$, $\xi_1^+$ and $\theta_1^+$ are associated by (d).

Now $\xi_1^+, \xi_1^-, \theta_1^+$ are associated by (a) and (e) because $\xi_1^-$ is independent of $\xi_1^+$ and $\theta_1^+$, and then $\xi_1 = \xi_1^+ + \xi_1^-$ and $\theta_1^+$ are also associated by (c). This concludes the proof of the lemma since $(\xi_n, \theta_n)_{n \ge 1}$ are i.i.d.
\end{proof}

The following Proposition~\ref{TAILS} describes the ``tails'' of $\xi_1$ and $\theta_1$. The proposition consists of two Parts (a) and (b). We stress that only Part (a), whose proof is straightforward, is used to prove Theorem~\ref{UPPER B} and Part 2 of Theorem~\ref{LOWER B}. The proof of Part 1 of Theorem~\ref{LOWER B} requires more complicated Corollary~\ref{COR XI} of Part (b). Although Part (b) itself is not used directly in the proofs of our main results, it is interesting because of its Corollary~\ref{COR HAAN} and because it generalizes the crucial Theorem~1 of Sinai~\cite{Sinai}.

Let $\xi_{ex}:= \int_0^1 W_{ex} (u) du$ be the area of a standard Brownian excursion. The latter is defined as $W_{ex} (u) := (\overline{\nu}-\underline{\nu})^{-1/2} \bigl|W(\underline{\nu} + u(\overline{\nu}-\underline{\nu}))\bigr|$, where $W(u)$ is a standard Brownian motion,  $\underline{\nu}$ is the last zero of $W(u)$ before $1$ and $\overline{\nu}$ is the first zero after $1$. For $x \ge 0$, put $$F(x) := \E \min \bigl \{ x^{-1/3} \xi_{ex}^{1/3}, 1 \bigr \}.$$ Clearly, $F(x)$ is decreasing, $F(0)=1$, and $F(\infty)=0$. By Janson~\cite{Janson}, $\xi_{ex}$ is continuous and has finite moments of any order, so $F(x)$ is continuous and, by $F(x)= x^{-1/3} \E \min \bigl \{ \xi_{ex}^{1/3}, x^{1/3} \bigr \} $, we have $\lim \limits_{x \to \infty} x^{1/3} F(x)=\E \xi_{ex}^{1/3} < \infty$.

\begin{prop}\label{TAILS}
Let $S_n$ be a centered random walk with a finite variance.

(a) $\theta_1^+$ belongs to the domain of normal attraction of a spectrally positive $\alpha$-stable law with exponent $1/2$, and the same holds for $\theta_1^0$ if $S_n$ is integer-valued.

(b) If $S_n$ satisfies assumptions of Theorem~\ref{LOWER B} or $S_n$ is upper exponential, then for any $s,t \ge 0$ such that $s+t>0$ it holds that
\begin{eqnarray}
&& \lim_{n \to \infty} n^{1/2} \P \bigl \{ \xi_1^+ > sn^{3/2}, \theta_1^+ > tn \bigr \}
= \lim_{n \to \infty} n^{1/2} \P \bigl \{ \xi_1^- < -sn^{3/2}, \theta_1^- > tn \bigr \} \notag \\
&=& \lim_{n \to \infty} n^{1/2} \P \bigl \{ \xi_1 > sn^{3/2}, \theta_1 > tn \bigr \} =
\lim_{n \to \infty} n^{1/2} \P \bigl \{ \xi_1 < -sn^{3/2}, \theta_1 > tn \bigr \} \notag \\
&=& C_{Law(S_1)} \, t^{-1/2} F(\sigma s t^{-3/2}) \label{2d tails},
\end{eqnarray}
where $C_{Law(S_1)}=\frac{(1-a_0) \E|S_1|}{\sqrt{2 \pi} a_+ a_- \sigma}$ or $C_{Law(S_1)} = \sqrt{\frac{2}{\pi}} \frac{\sigma}{\E|S_1|}$, respectively. The right-hand side of (\ref{2d tails}) at $t=0$ is defined by continuity. If $S_n$ is integer-valued and the lattice span of $S_1$ is $1$, then
\begin{equation} \label{2d tails0}
\lim_{n \to \infty} n^{1/2} \P \bigl \{ \xi_1^0 > sn^{3/2}, \theta_1^0 > tn \bigr \} = \lim_{n \to \infty} n^{1/2} \P \bigl \{ \xi_1^0 < -sn^{3/2}, \theta_1^0 > tn \bigr \} =  \frac{\sigma}{\sqrt{2 \pi t}} F(\sigma s t^{-3/2}).
\end{equation}
\end{prop}

\begin{cor}
\label{COR XI}
Suppose $S_n$ satisfies assumptions of Theorem~\ref{LOWER B}. Then $\xi_1$ belongs to the domain of normal attraction of a symmetric $\alpha$-stable law with exponent $1/3$.
\end{cor}

As an immediate consequence of de Haan et al.~\cite{de Haan+}, we have the following.

\begin{cor}
\label{COR HAAN}
Under conditions of Part (b) of Proposition~\ref{TAILS},
$$\Bigl (\frac{\xi_1 + \dots + \xi_n}{n^3}, \frac{\theta_1 + \dots + \theta_n}{n^2} \Bigr) \stackrel{\D}{\longrightarrow} (\xi, \theta),$$ where $Law(\theta)$ is spectrally positive $\alpha$-stable with exponent $1/2$ and $Law(\xi)$ is symmetric $\alpha$-stable with exponent $1/3$. The same holds for sums of $\xi_i^0$ and $\theta_i^0$.
\end{cor}

Before we get to the proofs, recall some important facts on ladder variables of random walks from Feller~\cite{Feller}. For any random walk $U_n$, define the first descending and ascending ladder epochs as $\tau_+:= \min\{k >0: U_k<0\}$ and $\tau_-:= \min\{k >0: U_k>0\}$, respectively, where by definition $\min_{\varnothing} := \infty$. We introduce such notations considering $\tau_+$ as the duration of the first {\it positive excursion} of $U_k$ (increased by one of course) rather than the first moment when $U_k$ becomes negative. It is readily seen that
\begin{equation}
\label{tau^+ -- min} \P \bigl \{ \tau_+ > n \bigr\} = \P \Bigl \{ \min \limits_{1 \le i
\le n} U_i  \ge 0 \Bigr \}.
\end{equation}

Denote $$c_+:=\sum_{n=1}^\infty \frac1n \bigl( \P \{ U_n > 0\} - 1/2 \bigr), \quad c_0:=\sum_{n=1}^\infty \frac1n \P \{ U_n = 0\} , \quad c_-:=\sum_{n=1}^\infty \frac1n \bigl( \P \{ U_n < 0\} - 1/2 \bigr)$$ if the sums are well-defined. If $c_+$ and $c_-$ are finite, then
\begin{equation}\label{ladder moments}
\lim_{n \to \infty} n^{1/2} \P \bigl \{ \tau_+ > n \bigr\} = \frac{e^{c_+ + \, c_0}}{\sqrt{\pi}}, \quad \lim_{n \to \infty} n^{1/2} \P \bigl \{ \tau_- > n \bigr\} = \frac{e^{c_- + \, c_0}}{\sqrt{\pi}}.
\end{equation}
It is known that $c_0$ is {\it always} finite while $c_+$ and $c_-$
are finite if $\E U_1 = 0$ and $0 < \V U_1 =: \sigma^2 < \infty$. Under the latter
conditions, we also have
\begin{equation}
\label{ladder heights}
\E U_{\tau_+} = -\frac{\sigma}{\sqrt{2}} e^{c_+ + c_0}, \quad \E
U_{\tau_-} = \frac{\sigma}{\sqrt{2}} e^{c_- + c_0}
\end{equation}
for the ladder heights $U_{\tau_+}$ and $U_{\tau_-}$. Finally, if $\P \{ U_n > 0\} \to
1/2$, then
\begin{equation} \label{tau_+ tail l}
\P \bigl \{ \tau_+ > n \bigr\} \sim n^{-1/2} L(n),
\end{equation}
for some function $L(n)$ that is slowly varying at infinity, see Rogozin~\cite{Rogozin}.

\begin{proof}[{\bf Proof of Proposition~\ref{TAILS}}]
I. The statements on $\xi_1^+$ and $\theta_1^+$.

\underline{Case 1}: $s=0$ and $t>0$. Without loss of generality, put $t=1$. We have
\begin{equation}\label{trivial}
\P \{\theta_1^+ >n\} = \widetilde{\P} \{\tau_+ >n+1\} = a_+^{-1} \P
\{\tau_+ >n+1, S_1>0 \} = a_+^{-1} \bigl (\P \{\tau_+ >n+1\} - a_0 \P \{\tau_+ >n\} \bigr),
\end{equation}
and by \eqref{ladder moments}, since $S_k$ is centered and has a finite variance,
\begin{equation}
\label{theta^+ as} \lim_{n \to \infty} n^{1/2} \P \{\theta_1^+ > n\} = \frac{1-a_0}{a_+}
\cdot \frac{e^{c_+ +c_0}}{\pi^{1/2}}.
\end{equation}
This relation proves Part (a) of the proposition.

To simplify the right-hand side of \eqref{theta^+ as}, write $\E |S_1| = 2 a_+ \E(S_1| S_1 >0)$, which follows from $\E S_1 =0$. Under assumptions of Part (b), $S_k$ is upper exponential, right-continuous, or upper geometric, so $\mbox{Law}(S_1| S_1 >0) = \mbox{Law}(S_{\tau_-})$, and recalling \eqref{ladder heights}, $\E |S_1| = 2 a_+ \E S_{\tau_-} = \sqrt{2} a_+ \sigma  e^{c_- + c_0}$. Then $e^{c_- + c_0} = \frac{\E |S_1|}{\sqrt{2} a_+ \sigma}$, and by $e^{c_+ + c_0 + c_-} =1$, we get $e^{c_+}= \frac{\sqrt{2} a_+ \sigma}{\E |S_1|}$. If $S_k$ is upper exponential, then clearly $\P\{S_k=0\}=a_0^k$, hence $e^{c_0}=\frac{1}{1-a_0}$, and from \eqref{theta^+ as} we have $C_{Law(S_1)} = \sqrt{\frac{2}{\pi}} \frac{\sigma}{\E|S_1|}$ for the constant in \eqref{2d tails}. If $S_k$ satisfies assumptions of Theorem~\ref{LOWER B}, by the same arguments as above, $e^{c_-}= \frac{\sqrt{2} a_- \sigma}{\E |S_1|}$. Now $e^{c_+ + c_0 + c_-} =1$ implies $e^{c_0} = \frac{(\E |S_1|)^2}{2 a_+ a_- \sigma^2}$, and from \eqref{theta^+ as}, $C_{Law(S_1)}=\frac{(1-a_0) \E|S_1|}{\sqrt{2 \pi} a_+ a_- \sigma}$.

\underline{Case 2}: $s \ge 0$ and $t > 0$. We state one important particular case of the result of Shimura~\cite{Shimura} on convergence of discrete excursions. Let $W(t)$ be a standard Brownian motion, and let $\bar{W}(t) := W(t) - \inf_{0 \le s \le t} W(s)$ be a reflecting Brownian motion. Then for any random walk $U_n$ such that $\E U_1 = 0$ and $0 < \V U_1 =: \sigma^2 < \infty$, for any $\varepsilon>0$
\begin{equation}
\label{Shimura} \mbox{Law} \Bigl( \Bigl(\frac{\tau_+}{n}, \frac{U_{\min\{\tau_+, [n
\cdot]\}}}{\sigma n^{1/2}} \Bigr) \Bigl | \Bigr. \tau_+ > \varepsilon
n\Bigr)\stackrel{\D}{\longrightarrow} \mbox{Law} \Bigl( \bigl(\nu_\varepsilon''-\nu_\varepsilon', \bar{W}
(\nu_\varepsilon' + \min\{\cdot, \nu_\varepsilon''-\nu_\varepsilon'\}) \bigr) \Bigr )
\end{equation}
in $\R \times \mathcal{D}[0, \infty)$ as $n \to \infty$, where $\mathcal{D}$ stands for Skorokhod space and $(\nu_\varepsilon', \nu_\varepsilon'')$ is the first pair of successive zeros of $\bar{W}$ such that $\nu_\varepsilon''-\nu_\varepsilon' >
\varepsilon$.

Since the r.v.'s $\nu_\varepsilon''-\nu_\varepsilon'$ and $\int_{\nu_\varepsilon'}^{\nu_\varepsilon''} \bar{W}(u) du$ are continuous, from \eqref{ladder moments} and \eqref{Shimura} we find that for any $s \ge 0$ and $t \ge \varepsilon$,
\begin{eqnarray}
&& \P \bigl \{ \xi_+ > sn^{3/2}, \tau_+ > tn \bigr \} \notag \\
&\sim& \P \{ \tau_+ > \varepsilon n \} \P \Bigl \{ \nu_\varepsilon''-\nu_\varepsilon' > t,
\int_{\nu_\varepsilon'}^{\nu_\varepsilon''} \bar{W}(u) du > \sigma s \Bigr \} \notag \\
&\sim& \frac{e^{c_+ + \, c_0}}{ (\pi \varepsilon n)^{1/2} } \P \Bigl \{
\nu_\varepsilon''-\nu_\varepsilon' > t, (\nu_\varepsilon'' - \nu_\varepsilon') \int_0^1
\bar{W}(\nu_\varepsilon' + u(\nu_\varepsilon'' - \nu_\varepsilon')) du > \sigma s \Bigr
\} \label{2d tail 1}
\end{eqnarray}
as $n \to \infty$, where $\xi_+:=\sum_{k=1}^{\tau_+ - 1} S_k$ and by definition,
$\Sigma_{\varnothing}:=0$.

We claim that, first, the process $W_{ex}^{(\varepsilon)}(\cdot) := (\nu_\varepsilon'' - \nu_\varepsilon')^{-1/2} \bar{W}(\nu_\varepsilon' + \cdot (\nu_\varepsilon'' - \nu_\varepsilon'))$ is a standard Brownian excursion $W_{ex}(\cdot)$ on $[0,1]$ and, second, $W_{ex}^{(\varepsilon)}(\cdot)$ is independent with $\nu_\varepsilon'' - \nu_\varepsilon'$. Recall
the definition $W_{ex} (\cdot) := (\nu''-\nu')^{-1/2} \bar{W}(\nu'+ \cdot(\nu''-\nu'))$, where $\nu'$ is the last zero of $\bar{W}(\cdot)$ before $1$ and $\nu''$ is the first zero after $1$. $W_{ex} (\cdot)$ is usually defined in terms of $|W(\cdot)|$ but we used that $\bar{W}(\cdot) \stackrel{\D}{=} |W(\cdot)|$.

Indeed, it is known (for instance, see Drmota and Marckert~\cite{Drmota}) that if $U_n$ is a simple random walk, then $$\mbox{Law} \Bigl( \frac{U_{[\tau_+ \cdot]}}{\sigma \tau_+^{1/2}}  \Bigl | \Bigr. \tau_+= n\Bigr) = \mbox{Law} \Bigl( \frac{U_{[n \cdot]}}{\sigma n^{1/2}}  \Bigl | \Bigr. \tau_+= n\Bigr) \stackrel{\D}{\longrightarrow} \mbox{Law} \Bigl( W_{ex} (\cdot) \Bigr)$$ in $\mathcal{D}[0, 1]$. Hence for any $a>0$ and any cylindrical set $\mathcal{A} \subset \mathcal{D}[0, 1]$ that is generated by the product of intervals (the latter ensures $\P \bigl\{ W_{ex} (\cdot) \in \partial \mathcal{A} \bigr \} = \P \bigl\{ W_{ex}^{(\varepsilon)} (\cdot) \in \partial \mathcal{A} \bigr \} = 0$),
\begin{equation} \label{ref1}
\P \Bigl \{ \frac{U_{[\tau_+ \cdot]}}{\sigma \tau_+^{1/2}} \in \mathcal{A}, \tau_+ > a n \Bigr \} = \P \{ \tau_+ > a n \} \Bigl( \P \bigl\{ W_{ex} (\cdot) \in \mathcal{A} \bigr \} + o(1) \Bigr ).
\end{equation}
On the other hand, \eqref{Shimura} yields $$\mbox{Law} \Bigl( \Bigl(\frac{\tau_+}{n}, \frac{U_{[\tau_+ \cdot]}}{\sigma \tau_+^{1/2}}  \Bigr) \Bigl | \Bigr. \tau_+ > \varepsilon n\Bigr) \stackrel{\D}{\longrightarrow}  \mbox{Law} \Bigl( \bigl(\nu_\varepsilon''-\nu_\varepsilon', W_{ex}^{(\varepsilon)}(\cdot) \bigr) \Bigr)$$ in $\R \times \mathcal{D}[0, 1]$. Hence if $ a \ge \varepsilon$, then
\begin{equation} \label{ref2}
\P \Bigl \{ \frac{U_{[\tau_+ \cdot]}}{\sigma \tau_+^{1/2}} \in \mathcal{A}, \tau_+ > a n \Bigr \} = \P \{ \tau_+ > \varepsilon n \} \Bigl( \P \bigl\{
\nu_\varepsilon'' - \nu_\varepsilon' > a, W_{ex}^{(\varepsilon)}(\cdot) \in \mathcal{A} \bigr \} + o(1) \Bigr ).
\end{equation}
Finally, comparing \eqref{ref1} and \eqref{ref2} and using \eqref{ladder moments}, we obtain
$$\Bigl( \frac{\varepsilon}{a} \Bigr)^{1/2} \P \bigl\{ W_{ex} (\cdot) \in \mathcal{A} \bigr \} = \P \bigl\{ \nu_\varepsilon'' - \nu_\varepsilon' > a, W_{ex}^{(\varepsilon)}(\cdot) \in \mathcal{A} \bigr \},$$ which implies $W_{ex}^{(\varepsilon)}(\cdot) \stackrel{\D}{=} W_{ex}(\cdot)$ and independence of $\nu_\varepsilon'' - \nu_\varepsilon'$ and $W_{ex}^{(\varepsilon)}(\cdot)$.

Now, since $\P \bigl \{ \nu_\varepsilon''-\nu_\varepsilon' > t \bigr \} = (\frac{\varepsilon}{t} )^{1/2}$ for $t \ge \varepsilon$, we rewrite \eqref{2d tail 1} as
\begin{eqnarray*}
\lim_{n \to \infty} n^{1/2} \P \bigl \{ \xi_+ > sn^{3/2}, \tau_+ > tn \bigr \} &=&
\frac{e^{c_+ + \, c_0}}{ 2 \pi^{1/2} } \int_t^\infty z^{-3/2} \P \Bigl \{ \int_0^1  W_{ex}(u) du > \sigma s z^{-3/2} \Bigr \} dz \\
&=& \frac{e^{c_+ + \, c_0}}{3 (\sigma s)^{1/3} \pi^{1/2}} \int_0^{\sigma s t^{-3/2}} v^{-2/3} \P \bigl \{ \xi_{ex} > v \bigr \} dv,
\end{eqnarray*}
where we changed variables and put $\xi_{ex}:=\int_0^1  W_{ex}(u) du$. For any $x > 0$, write
\begin{eqnarray*}
&& \frac{1}{3x^{1/3}}\int_0^x v^{-2/3} \P \bigl \{ \xi_{ex} > v \bigr \} dv = \P \bigl \{ \xi_{ex} > x \bigr \} -  \frac{1}{x^{1/3}} \int_0^x v^{1/3} d \, \P \bigl \{ \xi_{ex} \le v \bigr \} \\
&=& x^{-1/3} \E \min \bigl \{ \xi_{ex}^{1/3}, x^{1/3} \bigr \} = \E \min
\bigl \{ x^{-1/3} \xi_{ex}^{1/3}, 1 \bigr \} =: F(x).
\end{eqnarray*}
Then
$$\lim_{n \to \infty} n^{1/2} \P \bigl \{ \xi_+ > sn^{3/2}, \tau_+ > tn \bigr \} = \frac{e^{c_+ + \, c_0}}{(\pi t)^{1/2} } F(\sigma s t^{-3/2}),$$
and arguing as in \eqref{trivial},
$$\lim_{n \to \infty} n^{1/2} \P \bigl \{ \xi_1^+ > sn^{3/2}, \theta_1^+ > tn \bigr \} =
\frac{1-a_0}{a_+} \cdot \frac{e^{c_+ + \, c_0}}{(\pi t)^{1/2}} F(\sigma s t^{-3/2}).$$
We already explained above why the constant in the right-hand side has the required form.

\underline{Case 3}: $s>0$ and $t=0$. Since the right-hand side of \eqref{2d tails} at $t=0$ is defined by continuity and \eqref{2d tails} is already proved for $s,t >0$, we should check that $$\lim_{n \to \infty} n^{1/2} \P \{ \xi_1^+ > sn^{3/2} \} = \lim_{t \to 0} \lim_{n \to \infty} n^{1/2} \P \{ \xi_1^+ > sn^{3/2}, \theta_1^+ > t n \}.$$ By the law of total probability, it suffices to show $$\lim_{t \to 0} \limsup_{n \to \infty} n^{1/2} \P \{\xi_1^+ >sn^{3/2}, \theta_1^+ \le t n\} = 0.$$ But $$\P \{\xi_1^+ > sn^{3/2}, \theta_1^+ < t n\} \le \widetilde{\P} \{\max \limits_{1 \le k \le \tau_+ - 1} S_k >t^{-1} s n^{1/2}, \tau_+ < t n\} \le a_+^{-1} \P \{\max \limits_{1 \le k \le \tau_+ - 1} S_k >t^{-1} s n^{1/2}\},$$ where the second estimate was obtained as in \eqref{trivial}, and by definition, $\max_{\varnothing} := -\infty$. Now the required estimate follows from Theorem~2 of Simura~\cite{Shimura}.

II. The statements on $\xi_1^-$ and $\theta_1^-$.

If $S_n$ is upper exponential, simply use $(\xi_1^-, \theta_1^-) \stackrel{\D}{=} (-\xi_1^+, \theta_1^+)$ from Lemma~\ref{EXP SYMM} and the part of \eqref{2d tails} on $\xi_1^+$ and $\theta_1^+$ proven above. If $S_n$ satisfies assumptions of Theorem~\ref{LOWER B}, $(\theta_1^-, \xi_1^-)$ has the same distribution as $(\bar{\theta}_1^+, -\bar{\xi}_1^+)$, where the bar means that the walk $\bar{S}_n:=-S_n$ is considered. Since $\bar{S}_n$ satisfies assumptions of Theorem~\ref{LOWER B} if $S_n$ does, we use the part of \eqref{2d tails} on $\xi_1^+$ and $\theta_1^+$ proven above and $C_{Law(S_1)} = C_{Law(-S_1)}$.

III. The statements on $\xi_1$ and $\theta_1$.

We only consider the case $s=0$ letting, without loss of generality, $t=1$. The proof of the other cases is absolutely similar. Let us check that for $\theta_1 = \theta_1^+ + \theta_1^-$,
$$\lim_{n \to \infty} n^{1/2} \P \{\theta_1 > n\} = \lim_{n \to \infty} n^{1/2} \P \{\theta_1^+ > n\} + \lim_{n \to \infty} n^{1/2} \P \{\theta_1^- > n\}.$$
By standard arguments, it suffices to show that
\begin{equation}
\label{boring}
\lim_{n \to \infty} n^{1/2}\P \bigl \{\theta_1^+ > n, \theta_1^- > n \bigr \} = 0.
\end{equation}
Under assumptions of Theorem~\ref{LOWER B}, $\theta_1^+$ and $\theta_1^-$ are independent, and the statement is trivial.

Otherwise, consider an independent copy $S'_n$ of the walk $S_n$. For any $x \ge 0$, put $\tau'_-(x):= \min \{k \ge 1: \, S'_k > x \}$. Since $\theta_1^- = \max\{k \ge 1: \widetilde{S}_{\theta_1^+ + k} - \widetilde{S}_{\theta_1^+ + 1} \le -\widetilde{S}_{\theta_1^+ + 1} \}$, we have $\theta_1^- \stackrel{\D}{=} \tau'_- \bigl( -\widetilde{S}_{\theta_1^+ + 1} \bigr)$, and for any $M >0$,
$$\P \bigl \{\theta_1^+ > n, \theta_1^- > n \bigr \} \le
\P \bigl \{\theta_1^+ > n, \widetilde{S}_{\theta_1^+ + 1} < -M \bigr \} + \P \bigl \{\theta_1^+ > n \bigr \} \P \bigl \{ \tau_-(M) > n \bigr \}.$$
Arguing as in \eqref{trivial}, we get \eqref{boring} from \eqref{ladder moments} and $$\lim_{M \to \infty} \limsup_{n \to \infty} n^{1/2} \P \bigl \{\tau_+ > n, S_{\tau_+} < -M \bigr \} =0,$$ which follows from Lemma~4 in Eppel~\cite{Eppel}.

IV. The statements on $\xi_1^0$ and $\theta_1^0$.

It is well known (Spitzer~\cite[Sec. 32]{Spitzer}) that
\begin{equation}
\label{Spitzer}
\lim_{n \to \infty} n^{1/2} \P\{\theta_1^0 > n\} = \sqrt{\frac{2}{\pi}} \sigma
\end{equation}
for any integer-valued random walk with a finite variance. Then we find the asymptotics of the ``tail'' of $(\theta_1^0, \xi_1^0)$ exactly as the one of $(\theta_1^+, \xi_1^+)$, up to the following differences. First, we use \eqref{Spitzer} instead of \eqref{ladder moments}. Second, instead of referring to \eqref{Shimura}, use the result of Kaigh~\cite{Kaigh} that $\frac{U_{[n \cdot]}}{\sigma n^{1/2}}$ conditioned on $\{\theta_1^0 = n\}$ weakly converges to a signed Brownian excursion $\varrho W_{ex}(\cdot)$, where $\P\{\varrho=1\}= \P\{\varrho=-1\}=1/2$ and $\varrho$ is independent of $W_{ex}(\cdot)$. The additional assumption that $S_1$ has span $1$ is required to use the result of Kaigh~\cite{Kaigh}.
\end{proof}

\section{The upper bound} \label{Sec Upper}

1. $S_n$ is an upper exponential random walk.

Define $\nu := \min \bigl \{k > 0 : \, \xi_1 + \dots + \xi_k < 0 \bigr\}$. Then $$\xi_1 +
\dots + \xi_\nu = \sum_{i=1}^{\tau_1} \widetilde{S}_i + \dots + \sum_{i=\tau_{\nu-1}
+1}^{\tau_\nu} \widetilde{S}_i = \sum_{i=1}^{\tau_\nu} \widetilde{S}_i < 0$$ implying $\P
\{ \tau_\nu \le N\} \le \P \bigl \{ \min \limits_{1 \le k \le N} \sum_{i=1}^k
\widetilde{S}_i  < 0 \bigr \} = 1 - \tilde{p}_N$, hence
\begin{equation}\label{upper}
\tilde{p}_N \le \P \{ \tau_\nu >N\}.
\end{equation}
We stress that \eqref{upper} is true for every random walk, but the r.v.'s $\xi_i$ are
i.i.d. if $S_n$ is upper exponential (or, of course, if $S_n$ is integer-valued and either upper geometric or right-continuous).

By a Tauberian theorem (see Feller~\cite[Ch.~XIII]{Feller}), the asymptotics of $\P \{
\tau_\nu >N\}$ as $N \to \infty$ can be found if we know the behavior of the
generating function $\chi(t)$ of $\tau_\nu$ as $t \nearrow 1$: for any $p \in (0,1)$ and
$c>0$,
\begin{equation}\label{Tauberian}
\P \{ \tau_\nu >N\} \sim \frac{c}{\Gamma(p) N^{1-p}} \quad
\Longleftrightarrow \quad 1 - \chi(t) \sim c (1-t)^{1-p} .
\end{equation}

Let us first find the generating function of the joint distribution of $\nu$ and
$\tau_\nu.$ For any positive integer $k$ and $l$, $$\P \bigl \{ \nu = k, \tau_\nu = l
\bigr \} = \P \Bigl \{ \xi_1 \ge 0, \dots, \xi_1 + \dots + \xi_{k-1} \ge 0, \xi_1 + \dots + \xi_k < 0, \theta_1 + \dots + \theta_k = l \Bigr \}.$$ The r.v. $\nu$ is the first descending ladder epoch of the walk $\xi_1 + \dots + \xi_n$, and its generating function is described by the Sparre-Andersen theorem, see Feller~\cite[Ch. XII]{Feller}. Sinai~\cite{Sinai} (Lemma~3) gives the following straightening of this result: the generating function $$\chi(s,t):= \sum_{k,l \ge 1} \P \{ \nu = k, \tau_\nu = l \} s^k t^l$$ of the random vector $(\nu, \tau_\nu)$ satisfies
$$\ln \frac{1}{1-\chi(s,t)} = \sum_{k, l \ge 1} \frac{s^k t^l}{k} \P \Bigl \{ \xi_1 +
\dots + \xi_k < 0, \, \theta_1 + \dots + \theta_k = l \Bigr \}.$$

By Lemma~\ref{EXP SYMM}, for the generating function $\chi(t) := \chi(1, t)$ of
$\tau_\nu$ it holds
\begin{eqnarray}
\ln \frac{1}{1-\chi(t)} &=& \sum_{k, l \ge 1} \frac{t^l}{k} \P \Bigl \{ \xi_1 + \dots +
\xi_k < 0, \, \theta_1 + \dots + \theta_k = l \Bigr \} \notag \\
&=& \frac12 \sum_{k, l \ge 1} \frac{t^l}{k} \P \Bigl \{ \theta_1 + \dots + \theta_k = l
\Bigr \} \label{1 - xi 2}
\end{eqnarray}
Since $\theta_k$ are i.i.d., $$\sum_{k, l \ge 1} \frac{t^l}{k} \P \Bigl \{ \theta_1 +
\dots + \theta_k = l \Bigr \} = \sum_{k = 1}^\infty \frac{1}{k} \sum_{l = 1}^\infty t^l
\P \Bigl \{ \theta_1 + \dots + \theta_k = l \Bigr \} = \sum_{k = 1}^\infty \frac{1}{k}
\zeta^k(t) = \ln \frac{1}{1- \zeta(t)},$$ where $\zeta(t)$ is the generating function of
$\theta_1$. Then
\begin{equation}\label{xi -- zeta}
1-\chi(t) = \sqrt{1-\zeta(t)},
\end{equation}
and using Part (a) of Proposition~\ref{TAILS} and the Tauberian theorem~\eqref{Tauberian} twice, we get $\P \{ \tau_\nu >N\} \sim c N^{-1/4}$. By \eqref{p eqv p} and \eqref{upper}, the upper bound follows.

Case 2. $S_n$ is an integer-valued random walk.

We argue exactly as in the proof of the first part. Replacing everywhere $\xi_n$ and
$\theta_n$ by $\xi_n^0$ and $\theta_n^0$, respectively,  we get $p_N \le \P \{
\tau_{\nu^0}^0 >N\}$ instead of \eqref{upper} and $$1-\chi^0(t) = \sqrt{1-\zeta^0(t)}
e^{H(t)}$$ instead of \eqref{xi -- zeta}, where $$H(t):= \frac12 \sum_{k, l \ge 1}
\frac{t^l}{k} \P \Bigl \{ \xi_1^0 + \dots + \xi_k^0 = 0, \, \theta_1^0 + \dots +
\theta_k^0 = l \Bigr \}$$ emerges in the analogue of \eqref{1 - xi 2}. The limit $\lim
\limits_{t \to 1} H(t)$ exists and is finite because $H(t)$ is increasing and the series
$$H(1) = \sum_{k, l \ge 1} \frac1k \P \Bigl \{ \xi_1^0 + \dots + \xi_k^0 = 0, \,
\theta_1^0 + \dots + \theta_k^0 = l \Bigr \} = \sum_{k=1}^\infty \frac1k \P \Bigl \{
\xi_1^0 + \dots + \xi_k^0 = 0 \Bigr \}=c_0$$ is convergent for any random walk. Hence the upper bound follows from Part (a) of Proposition~\ref{TAILS} and the Tauberian theorem
\eqref{Tauberian} as above.

\section{The lower bound} \label{Sec Lower}
By \eqref{main estimate'}, we estimate
\begin{eqnarray*}
\widetilde{\P}\Bigl \{ \min \limits_{1 \le k \le N} \sum_{i=1}^k S_i  \ge 0 \Bigr \}
&\ge& \P \Bigl \{ \min \limits_{1 \le k \le \sqrt{N}} \sum_{i=1}^k \xi_i  \ge 0, \, \eta(N)+1 \le \sqrt{N} \Bigr \} \\
&=& \P \Bigl \{ \min \limits_{1 \le k \le \sqrt{N}} \sum_{i=1}^k \xi_i  \ge 0, \, \theta_1 + \cdots + \theta_{\sqrt{N}} > N \Bigr \}  \\
&\ge& \P \Bigl \{ \min \limits_{1 \le k \le \sqrt{N}} \sum_{i=1}^k \xi_i  \ge 0, \,
\theta_1^+ + \cdots + \theta_{\sqrt{N}}^+ > N \Bigr \} .
\end{eqnarray*}
By Lemma~\ref{ASSOC} and sufficient condition of association (c),
\begin{eqnarray*}
\widetilde{\P}\Bigl \{ \min \limits_{1 \le k \le N} \sum_{i=1}^k S_i  \ge 0 \Bigr \} &\ge& \P \Bigl \{ \min \limits_{1 \le k \le \sqrt{N}} \sum_{i=1}^k \xi_i  \ge 0 \Bigr \} \cdot \P \Bigl \{  \theta_1^+ + \cdots + \theta_{\sqrt{N}}^+ > N \Bigr \}\\
&\ge& c\, \P \Bigl \{ \min \limits_{1 \le k \le \sqrt{N}} \sum_{i=1}^k \xi_i \ge 0 \Bigr \}
\end{eqnarray*}
for some $c>0$ and all $N$, were we used Part (a) of Proposition~\ref{TAILS} for the last line.

Under assumptions of Part 2 of Theorem~\ref{LOWER B}, the distribution of $\xi_1$ is symmetric, see Lemma~\ref{EXP SYMM} for the case of two-sided exponential walks. Hence for the random walk $\sum_{i=1}^k \xi_i$ we have $c_+ = -c_0/2$, which is always finite, and Part 2 of Theorem~\ref{LOWER B} follows from \eqref{p eqv p}, \eqref{tau^+ -- min}, and \eqref{ladder moments}.

The proof of Part 1 of Theorem~\ref{LOWER B}, actually, takes much more efforts
because it requires the use of Corollary~\ref{COR XI} of Proposition~\ref{TAILS}. The latter implies that $\P \{ \xi_1 + \dots + \xi_n > 0\} \to 1/2$. Unfortunately, we can not verify that the series
\begin{equation} \label{series}
\sum_{n=1}^\infty \frac1n \bigl( \P \{ \xi_1 + \dots + \xi_n > 0\} - 1/2 \bigr)
\end{equation}
converges, and we should use \eqref{tau_+ tail l} instead of \eqref{ladder moments}.

Convergence of series of the type \eqref{series} was studied by Egorov~\cite{Egorov}, who considered rates of convergence in stable limit theorems and stated his results exactly
in the form of \eqref{series}. It is, however, unclear how to check his conditions for our case. A proof of the convergence would eliminate the slowly varying factor $l(N)$ in
Theorem~\ref{LOWER B}.

\section{Open questions and concluding remarks} \label{Sec Concluding}
1. Obtaining the lower bound under less restrictive conditions.

The most restrictive assumptions of Theorem~\ref{LOWER B} are the ones imposed on $\mbox{Law} (S_1 | S_1<0)$. We used these assumptions {\it only} in the proof of association of $\xi_1$ and $\theta_1^+$. It seems that these variables are associated under much less restrictive conditions and, possibly, under no assumptions at all. Simulations show that association holds in many cases. Note that the direct use of sufficient condition of association (c) is impossible because $\xi_1$ is {\it not} a
coordinate-wise increasing function of associated r.v.'s $\widetilde{S}_1,
\widetilde{S}_2, \dots$

2. Elimination of the slowly varying term in Theorem~\ref{LOWER B}.

As we explained above, the slowly varying factor could be eliminated if we show that the series \eqref{series} is convergent. The rate of convergence in stable limit theorems is usually estimated under existence of so-called pseudomoments of $\xi_1$. The pseudomoment of $\xi_1$ of order $1/3$ exists if the functions $x^{1/3} \P \{ \xi_1 > x\}$ and $x^{1/3} \P \{ \xi_1 < -x\}$ have a regular behavior as $x \to \infty$. It seems that the ``tails'' of $\xi_1$ could be controlled if we had an appropriate rate of convergence of discrete excursions to a Brownian excursion. We know only one result on this question: Drmota and Marckert~\cite{Drmota} gives the rate of convergence of positive excursions of left-continuous random walks. Since we need rates for both positive and negative excursions, the only slackened random walks would be covered, giving no refinement to Theorem~\ref{LOWER B}.

3. When the first draft of this paper was already written, the author became aware that Frank Aurzada and Steffen Dereich were also working on one-sided small deviation probabilities of integrated random processes, and they considered $p_N$ as a particular case. The methods of their paper~\cite{Germans} are entirely different from the ones presented here.

\section*{Acknowledgements}
A part of this work was done during the visit of the author to the University Paris 12
Val de Marne. The author thanks the University and his host Marguerite Zani for care and hospitality. He is also grateful to Mikhail Lifshits and Wenbo Li for their attention to the paper and to Vidmantas Bentkus and Vladimir Egorov for discussions on rates of convergence in stable limit theorems. Finally, the author thanks the anonymous referee for comments and useful suggestions.

\end{document}